\documentclass[11pt]{article}
\usepackage{amssymb, amsmath,amsthm,array}
\usepackage{float}
\usepackage{latexsym,shortvrb}
\usepackage[english]{babel}
\usepackage[dvips]{graphics, epsfig, color}
\usepackage{multirow}
\usepackage{rotating}
\usepackage{mathrsfs}
\usepackage{verbatim}
\usepackage{algorithm}
\usepackage{algorithmic}
\usepackage{float}  
\usepackage{lipsum}
\makeatletter
\newenvironment{breakablealgorithm}
{
		\begin{center}
			\refstepcounter{algorithm}
			\hrule height.8pt depth0pt \kern2pt
			\renewcommand{\caption}[2][\relax]{
				{\raggedright\textbf{\ALG@name~\thealgorithm} ##2\par}%
				\ifx\relax##1\relax 
				\addcontentsline{loa}{algorithm}{\protect\numberline{\thealgorithm}##2}%
				\else 
				\addcontentsline{loa}{algorithm}{\protect\numberline{\thealgorithm}##1}%
				\fi
				\kern2pt\hrule\kern2pt
			}
		}{
		\kern2pt\hrule\relax
	\end{center}
}
\makeatother

\usepackage{booktabs}

\usepackage{hyperref}
\usepackage{float}
\usepackage{mathrsfs}
\usepackage{graphicx}
\usepackage{epsfig}
\usepackage{tikz}
\usetikzlibrary{angles,quotes}
\newtheorem{theorem}{Theorem}[section]

\newtheorem{proposition}  {Proposition}[section]
\newtheorem{assumption} {Assumption}[section]

\newtheorem{remark} {Remark}[section]

\usepackage{cases}
\newcommand\bcase{\begin{numcases}{}}
\newcommand\ecase{\end{numcases}}

\begin{document}
	\title{Efficient Algorithm for QCQP problem with Multiple Quadratic Constraints}
	\author{Yin H. \thanks{Yin H. is the student of the College of Science, Zhejiang Sci-Tech University,
			Hangzhou,  China, 310018.  E-mail: hs970233@163.com.
			}%
		}
\maketitle
\begin{abstract}
	Starting from a classic financial optimization problem, we first propose a cutting plane algorithm for this problem. Then we use spectral decomposition to tranform the problem  into an equivalent D.C. programming problem, and the corresponding upper bound estimate is given by the  SCO algorithm; then the corresponding lower bound convex relaxation is given by McCormick envelope. Based on this, we propose a global algorithm for this problem and establish the convergence of the algorithms. What's more, the algorithm is still valid for QCQP with multiple quadratic constraints and quadratic matrix in general form.
\end{abstract}
\section{Problem reformulation}
\subsection{Simple introduction to the program}

	 Brown\cite{Brown2010} analyzed the portfolio strategies when facing uncertain liquidity risks and proposed the corresponding portfolio optimization model. What's more, he eplored further by introducing the shock after the first-period trading and remains the second-period problem unsolved. Let's first review the background of the problem.
	 
	  Consider the continuous trading model in a short term. $Y_t\in \mathbb{R}^n,t\in[0,\tau]$ is the rate that the investor trades the assets. The investor's holdings are denoted by $X_t\in\mathbb{R}^n, X_t=x_0+\int_{0}^{t}Y_s ds$. It's plain to assume that $x_0>0$ and $Y_t$ is an $L^2-$ function.
	
	The price of the assets at time t can be given by $P_t\in\mathbb{R}^n$,
	\begin{equation}
		P_t=q+\Gamma X_t+\Lambda Y_t
	\end{equation}
	
	The first term, $q\in\mathbb{R}^n$, the intercept of the linear model, account for the present value of underlying future cash flows. The second and third terms partition the price impact of trading into permanent and temporary components. The permanent component measures the change in price that is independent of the trading rate and is affected by the investor's holdings. The temporary component 
	reflects real-time price changes brought by trading. Both $\Gamma\in \mathbb{R}^n$ and $\Lambda\in\mathbb{R}^n$ are diagonal matrices, the diagonal entries of which are the price impacts of asset $i$. We assume that the components of $\Lambda$ and $\Gamma$ are positive.
	
	After the first-period trade, the leverage crisis has been resolved and the investor still has some chips to chase the potential profit. However, because of the unforeseen equity withdrawals or higher cash requirements, there remains uncertainty about the need for further liquidity during the second period. We model the shock as an early equity withdrawal;and assume that the amount withdrawn is a Bernoulli random variable $\Delta$ such that
	\begin{equation*}
		\Delta = \left\{\begin{aligned}
			\delta,&\quad probablity\quad\pi,\\
			0,&\quad probablity\quad 1-\pi.
		\end{aligned}\right.
	\end{equation*}

	Before moving towards to detail optimization problem, we shall discuss some technical assumptions on the parameters of the problem.
	\begin{assumption}
		The primal problem is strictly feasible and $\rho_1e_0-l_0<0$.
	\end{assumption}

	It's plain to have such an assumption, which accounts for that the leverage decreasing is immediate. The investor should make decisions as soon as possible to save the day.

	The two-period model and details are reformulated as
	\begin{equation}\label{ini}
	\begin{aligned}
		\max\quad&E_\Delta e_2=(1-\pi)e_1+\pi e_2\\
		s.t.\quad&\rho_1e_1-l_1\geq0,\\
		&\rho_2e_2-l_2\geq0,\\
		&-x_0\leq y_1\leq0,\\
		&-x_1\leq y_2\leq0.
	\end{aligned}
\end{equation}
where the expectation of the objective function is 
	\begin{equation*}
	\begin{aligned}
		E_\Delta e_2&=(1-\pi)e_1+\pi e_2\\
		&=e_0-\pi\delta+\begin{bmatrix}
			\Gamma x_0\\
			\pi\Gamma x_0
		\end{bmatrix}^t\begin{bmatrix}
			y_1\\
			y_2
		\end{bmatrix}-\begin{bmatrix}
			y_1\\
			y_2
		\end{bmatrix}^t\begin{bmatrix}
			\Lambda-\frac{1}{2}\Gamma &-\frac{\pi}{2}\Gamma\\
			-\frac{\pi}{2}\Gamma &\pi(\Lambda-\frac{1}{2}\Gamma)
		\end{bmatrix}\begin{bmatrix}
			y_1\\
			y_2
		\end{bmatrix}
	\end{aligned}
\end{equation*}
The other variables and notations can be shown as
	\begin{equation*}
	\begin{aligned}
		l_1&=l_0+p^t_0y_1+y_1^T(\Lambda+\frac{1}{2}\Gamma)y_1,\\
		e_1&=e_0+x_0^T\Gamma y_1-y_1^T(\Lambda-\frac{1}{2}\Gamma)y_1.
	\end{aligned}
\end{equation*}
\begin{equation*}
	\begin{aligned}
		l_2&=l_0+\Delta+\begin{bmatrix}
			p_0\\
			p_0
		\end{bmatrix}^T\begin{bmatrix}
			y_1\\
			y_2
		\end{bmatrix}+\begin{bmatrix}
			y_1\\
			y_2
		\end{bmatrix}^T\begin{bmatrix}
			\Lambda+\frac{1}{2}\Gamma &\frac{1}{2}\Gamma\\
			\frac{1}{2}\Gamma & \Lambda+\frac{1}{2}\Gamma
		\end{bmatrix}\begin{bmatrix}
			y_1\\
			y_2
		\end{bmatrix},\\
		e_2&=e_0-\Delta+\begin{bmatrix}
			\Gamma x_0\\
			\Gamma x_0
		\end{bmatrix}^T\begin{bmatrix}
			y_1\\
			y_2
		\end{bmatrix}-\begin{bmatrix}
			y_1\\
			y_2
		\end{bmatrix}^T\begin{bmatrix}
			\Lambda-\frac{1}{2}\Gamma & -\frac{1}{2}\Gamma\\
			-\frac{1}{2}\Gamma & \Lambda-\frac{1}{2}\Gamma
		\end{bmatrix}\begin{bmatrix}
			y_1\\
			y_2
		\end{bmatrix}.
	\end{aligned}\\
\end{equation*}

	It's easy to see that the problem is a quadratic constrained quadratic programming(QCQP). Brown has made a lot of progress under convexity assumption;Chen present a Lagrangian algorithm for the first-period quadratic problem. Luo has considerd more general situation by generalizing the price matrix.

\subsection{Optimality condition}
	In this section, we derive the optimality condition by first-order optimal conditions without convexity assumptions\cite{Chen2014}. The results with convexity assumption can be referred in Brown.
	\begin{proposition}\label{optimal}
		The second leverage constraint of problem $(\ref{ini})$ is active at its optimal solution.
	\end{proposition}
	\begin{proof}
		Since the feasible set of problem $(\ref{ini})$ is compact and our model is continuous, there exists an optimal solution $y^*$. On the contrary, we assume that the second leverage constraint is not active at the optimal solution. Denote $y^*=(y_1^*,y^*_2)^T, y_1^*=(y^*_{1,1},y^*_{1,2},...,y^*_{1,m})^T,y_2^*=(y^*_{2,1},y^*_{2,2},...,y^*_{2,m})^T$, by the first order optimality condition, there exists $\mu^*=(\mu^*_0,\mu^*_1,...,\mu^*_{2m})\geq0,\nu^*=(\nu^*_0,\nu^*_1,...,\nu^*_{2m})\geq0$,for $i=1,2,...,m$ satisfying that
			\begin{align}
				\mu_0^*(\rho_1e_1(y^*_1)-l_1(y_1^*))&=0,\label{1st}\\
				\mu_i^*y^*_{1,i}&=0,\\
				\mu^*_{m+i}(y^*_{1,i}+x_{0,i})&=0,\\
				\nu^*_0(\rho_2e_2(y^*)-l_2(y^*))&=0,\\
				\nu^*_i(y^*_{1,i}+y^*_{2,i})&=0,\label{vi}\\
				\nu^*_{m+i}(y^*_{1,i}+y^*_{2,i}+x_{0,i})&=0,\label{vmi}\\
				\begin{split}
					-\nabla e_2(y^*_1)+\mu_0^*\nabla g_0(y^*_1)+\sum_{i=1}^{m}(\mu^*_i\nabla g_{0,i}(y^*_1)+\mu^*_{m+i}\nabla g_{0,m+i}(y^*_1))+\\ \nu^*_0\nabla g_1(y^*_1)+\sum_{i=1}^{m}(\nu^*_i\nabla g_{1,i}(y^*_1)+\nu^*_{m+i}\nabla g_{1,m+i}(y^*_1))&=0,
				\end{split}\label{gradient1}\\
				-\nabla e_2(y^*_2)+v^*_0\nabla g_1(y^*_2)+\sum_{i=1}^{m}(\nu^*_i\nabla g_{1,i}(y^*_1)+\nu^*_{m+i}\nabla g_{1,m+i}(y^*_1)) &=0\label{gradient2}.
			\end{align}
		where
		\begin{equation*}
			\begin{aligned}
				g_0(y)&=l_1(y)-\rho_1e_1(y), g_{0,i}(y)=y_{1,i},g_{0,m+i}(y)=-y_{1,i}-x_{0,i},i=1,...,m,\\
				g_1(y)&=l_2(y)-\rho_2e_2(y),g_{1,i}=y_{1,i}+y_{2,i},g_{1,m+i}=-y_{1,i}-y_{2,i}-x_{0,i},i=1,...,m.
			\end{aligned}
		\end{equation*}
	By the assumption that the second leverage constraint is not active, we have $\nu^*_0=0$. For $i=1,2,...m$, the gradient conditions (\ref{gradient1}) and (\ref{gradient2}) become
	\begin{align}
		\begin{split}
			2(\lambda_i-\frac{1}{2}\gamma_i)y_{1,i}-\pi\gamma_iy_{2,i}-\gamma_ix_{0,i}+\mu^*_0\nabla g_0(y^*_{1,i})+\mu_i^*-\mu_{m+i}^*\\
			+\nu_i^*-\nu_{m+i}^*&=0,
		\end{split}\label{grad1}\\
		\begin{split}
			2\pi(\lambda_i-\frac{1}{2}\gamma_i)y_{2,i}-\pi\gamma_iy_{1,i}-\pi\gamma_ix_{0,i}
			+\nu_i^*-\nu_{m+i}^*&=0
		\end{split}\label{grad2}
	\end{align}
	By (\ref{vi}) (\ref{vmi}) and $x_0>0$, $\nu^*_i,\nu^*_{m+i},i=1,...,m$ can not be positive simultaneously, we consider the following cases:
	\begin{itemize}
		\item[1.] $\nu^*_i=0,\nu^*_{m+i}>0.$
		By condition (\ref{vmi}), we have 
		\begin{equation*}
			y_{1,i}^*+y_{2,i}^*+x_{0,i}=0,i=1,...,m.
		\end{equation*}
	From (\ref{grad2}), we reach 
	\begin{equation*}
		\begin{aligned}
			\nu_{m+i}^*&=2\pi\lambda_iy_{2,i}-\pi\gamma_i(y_{1,i}^*+y_{2,i}^*+x_{0,i})\\
			&=2\pi\lambda_i y_{2,i}\\
			&\leq0.
		\end{aligned}
	\end{equation*}
	which contradicts to requirement of positiveness.
	\item[2.] $\nu^*_i=0,\nu^*_{m+i}=0.$
	From equation (\ref{grad2})
	\begin{equation*}
		2\lambda_i y_{2,i}=\gamma_i(y_{1,i}+y_{2,i}+x_{0,i})
	\end{equation*}
	\begin{itemize}
		\item 	$y_{1,i}+y_{2,i}+x_{0,i}>0$, obviously contradictory.
		\item	$y_{1,i}+y_{2,i}+x_{0,i}=0, \forall i=1,2,...m$
		we have$y_{1,i}=-x_{0,i},y_{2,i}=0,\mu_i^*=0$.
		Consider the first leverage constraint and the corresponding libilities $l_1$ and the equity $e_1$
		\begin{equation*}
			\begin{aligned}
				\mu_{m+i}^*&=2\lambda_i y_{1,i}+\mu_0^*\nabla g_0(y_{1,i}^*),\\
				l_1+e_1&=l_0+e_0-p_0^Tx_0=0.
			\end{aligned}
		\end{equation*}
		Then we can turn the first leverage constraint into
		\begin{equation*}
			\begin{aligned}
				\rho_1e_1-l_1&=(\rho_1+1)e_1\\
				&=(\rho_1+1)[e_0-x_0^T(\Lambda+\frac{1}{2}\Gamma)]\\
				&\geq0.
			\end{aligned}
		\end{equation*}
	\begin{itemize}
		\item $\rho_1e_1-l_1=0$.
		
		 We have $e_1=0,l_1=0,x_1=0$, i.e. after the first-period trade, the investor is debt free and has no chips to continue the second-period trade.
		\item $\rho_1e_1-l_1>0$.
		
		 By (\ref{1st}), $\mu_0^*=0$, the equation (\ref{grad1}) can be simplified into
		 \begin{equation*}
		 	\mu_{m+i}^*=2\lambda_iy_{1,i}<0
		 \end{equation*}
	 which contradicts to the non-negativeness of the coefficients.
	\end{itemize}
	\end{itemize}
	\end{itemize}
	\end{proof}
\subsection{Measure of Shock Level}

	With the proposition \ref{optimal}, we can turn to the shock level to find some interesting properties.
	
	From $\rho_2e_2-l_2\geq0$, we have
	\begin{equation*}
		\begin{aligned}
			\delta\leq\frac{\rho_2e_0-l_0+G(y_1,y_2)}{\rho_2+1}
		\end{aligned}
	\end{equation*}
	then $\delta$ can reach its maximum by
	\begin{equation*}
		\delta^* = \frac{\rho_2e_0-l_0+G^*(y_1,y_2)}{\rho_2+1}
	\end{equation*}
	where $G^*(y_1,y_2)$ equals to the optimal value of the following problem:
	\begin{equation*}
		\begin{aligned}
			\max\quad &G(y_1,y_2)=\begin{bmatrix}
				\rho_2\Gamma x_0-p_0\\
				\rho_2\Gamma x_0-p_0
			\end{bmatrix}^T\begin{bmatrix}
				y_1\\
				y_2
			\end{bmatrix}-\\
			&\begin{bmatrix}
				y_1\\
				y_2
			\end{bmatrix}^T\begin{bmatrix}
				\rho_2(\Lambda-\frac{1}{2}\Gamma)+(\Lambda+\frac{1}{2}\Gamma) &\frac{1-\rho_2}{2}\Gamma\\
				\frac{1-\rho_2}{2}\Gamma & \rho_2(\Lambda-\frac{1}{2}\Gamma)+(\Lambda+\frac{1}{2}\Gamma)
			\end{bmatrix}\begin{bmatrix}
				y_1\\
				y_2
			\end{bmatrix}\\
			s.t.\quad&\rho_1e_1-l_1\geq0,\\
			&-x_0\leq y_1\leq0,\\
			&-x_0\leq y_1+y_2\leq0.
		\end{aligned}
	\end{equation*}
	
	\begin{assumption}
	 	The first-period trade $y_1=-x_0/2$ generates sufficient cash to meet the first-period leverage constraint.
	\end{assumption}
		We followed the assumption of Brown\cite{Brown2010}. The two-period trade is extreme, so it's reasonable to make such an assumption. What'more, half sale is very common in trading. If not satisfied, the investor is exhausted after fire sale during the first period and he may not have enough chips to cope with the second period trade.
	
		As we can see, the assumption with the form one can be explained as 
	\begin{equation*}
		\rho_1e_1-l_1|_{y_1=0}<0,\\
		\rho_1e_1-l_1|_{y_1=-\frac{x0}{2}}>0.
	\end{equation*}
	we take some notations to simply the above programming.
	\begin{equation*}
		\begin{aligned}
			\max\quad&G(y_1,y_2)=\begin{bmatrix}
				\rho_2\Gamma x_0-p_0\\
				\rho_2\Gamma x_0-p_0
			\end{bmatrix}^T\begin{bmatrix}
				y_1\\
				y_2
			\end{bmatrix}-\begin{bmatrix}
				y_1\\
				y_2
			\end{bmatrix}^T\begin{bmatrix}
				T & S\\
				S & T
			\end{bmatrix}\begin{bmatrix}
				y_1\\
				y_2
			\end{bmatrix}\\
			s.t.\quad&\rho_1e_1-l_1\geq0,\\
			&-x_0\leq y_1\leq0,\\
			&-x_0\leq y_1+y_2\leq0.
		\end{aligned}
	\end{equation*} 
	where
	\begin{flalign*}
		T=\rho_2(\Lambda-\frac{1}{2}\Gamma)+(\Lambda+\frac{1}{2}\Gamma), S=\frac{1-\rho_2}{2}\Gamma.
	\end{flalign*}
	It seems the key is the convexity of the objective function. Actually, when convex, it's easy to have optimal $y^*=(y_1^*,y_2^*)=(-\frac{x_0}{2},-\frac{x_0}{2})$; While removing the convexity assumption, the feasible region is still compact, so the optimal lies on the boundary. Substitute the boundary constraints into the objective function, we have
	\begin{equation}\label{eq:T-S}
		\begin{aligned}
			y^*&=\arg\max-\begin{bmatrix}
				y_1\\
				-x_0-y_1
			\end{bmatrix}^T\begin{bmatrix}
				T &S \\
				S &T
			\end{bmatrix}\begin{bmatrix}
				y_1\\
				-x_0-y_1
			\end{bmatrix}+(\rho_2\Gamma x_0-p_0)^T(-x_0)\\
			&=\arg\max -y_1^T(T-S)y_1-y_1^T(T-S)x_0.
		\end{aligned}
	\end{equation}
	It's easy to verify the matrix $T-S$ is positive definite, so we still have the optimal as $y^*=(-\frac{x_0}{2},-\frac{x_0}{2})$. By the assumption, the first-period leverage constraint is not binding at the point, let's prove the optimality next.
	
	Consider the optimal problem as below
	\begin{equation*}
		\begin{aligned}
			\max\quad&G(y_1,y_2)\\
			s.t.\quad& y_1+y_2=\alpha,\alpha\in[-x_0,0].
		\end{aligned}
	\end{equation*}
	As $\alpha$ range from $[-x_0,0]$, the optimal problem is equal to the primal problem when the optimal solution lies in the feasion region of the primal problem. Substitute the last constraints into the objective function and solve the unconstrained optimization problem we get $y^*_\alpha=(\frac{\alpha}{2},\frac{\alpha}{2})$, so what remains is to prove the following:
	\begin{equation*}
		\delta(-\frac{x_0}{2},-\frac{x_0}{2})\geq\delta(\frac{\alpha}{2},\frac{\alpha}{2}),\forall\alpha\in[-x_0,0].
	\end{equation*}
	Benefit from the special structure of the matrix $\begin{bmatrix}
		T &S\\ S &T
	\end{bmatrix}$, and after some calculation, we arrive at
	\begin{equation*}
		(\alpha+x_0)^T[\frac{1}{2}(T+S)(\alpha-x_0)-(\rho_2\Gamma x_0-p_0)]\geq0,\forall\alpha\in[-x_0,0].
	\end{equation*}
	Recall that $T$ and $S$ are diagonal matrix, by discussing the linear terms, we give the sufficient conditions for the above inequation to hold
	\begin{itemize}
		\item $\rho_2(\lambda_i-\gamma_i)+(\lambda_i+\gamma_i)>0, p_{0,i}>(\rho_2+1)\lambda_i x_{0,i}+\gamma_i x_{0,i},$
		\item $\rho_2(\lambda_i-\gamma_i)+(\lambda_i+\gamma_i)=0, p_{0,i}>\rho_2\gamma_i x_{0,i},$
		\item $\rho_2(\lambda_i-\gamma_i)+(\lambda_i+\gamma_i)<0,p_{0,i}>\frac{1}{2}(\rho_2+1)(\lambda_i+\gamma_i)x_{0,i}.$
	\end{itemize}
	
	With the conditions above, it's sufficient to conclude that $y^*=(-\frac{x_0}{2},-\frac{x_0}{2})$ will be the optimal without the convexity assumption.

\section{Initialization and Cut-Plane Algorithm}
	Before we move towards further research, let's turn to get an initial feasible point by using cut-plane algorithm.
	 
For convenience, the above QCQP programming (\ref{ini}) can be abbreviated as
	\begin{equation}\label{P}\tag{P}
	\begin{aligned}
		\min\quad& f(y)\\
		s.t.\quad& g(y)\leq0,\\
		& h(y)\leq0.\\
		& y\in D.
	\end{aligned}
\end{equation}
	
	Consider the lagrangian dual problem
	\begin{equation*}
		\begin{aligned}
			\max\quad& \theta(t_1,t_2)\\
			s.t.\quad&(t_1,t_2)\geq0.\\
			\theta(t_1,t_2) &= \inf\{f(y)+t_1g(y)+t_2h(y)|y\in D\}.
		\end{aligned}
	\end{equation*}
	
	The above problem can be reformulated as 
		\begin{equation*}\tag{D}\label{D}
		\begin{aligned}
			z^* = \max\quad &z\\
			s.t.\quad & z\leq f(y)+t_1g(y)+t_2g(y),\forall y\in D,\\
			&(t_1,t_2)\geq0.
		\end{aligned}
	\end{equation*}
	 
	 Suppose there exists $\{y^i\}^{k-1}_{i=0}\in D$, consider the following semi-infinite programming
	 \begin{equation}\tag{CD}\label{CD}
	 	\begin{aligned}
	 		\max\quad & z\\
	 		s.t.\quad & z\leq f(y^i)+t_1g(y^i)+t_2h(y^i),i=0,1,...,k-1\\
	 		&(t_1,t_2)\geq0.
	 	\end{aligned}
	 \end{equation}
 	Let $(z^k,t_1^k,t_2^k)$ be an optimal solution to the above approximating problem. If the solution satisfies (\ref{D}), then it is an optimal solution to the lagrangian dual problem. To satisfy this, consider the 
 	following subproblem 
 	\begin{equation}\label{1}
 		\begin{aligned}
 			\min\quad&f(y)+t_1^kg(y)+t_2^kh(y)\\
 			s.t.\quad&y\in D.
 		\end{aligned}
 	\end{equation}
	Let $y^k$ be the optimal to the above problem, so that
	\begin{equation*}
		\theta(t_1^k,t_2^k)= f(y^k)+t_1^kg(y^k)+t_2^kh(y^k)
	\end{equation*}
	If $z^k\leq\theta(t_1^k,t_2^k)$, then $(t_1^k,t_2^k)$ is an optimal solution to lagrangian dual problem (\ref{D}). Otherwise, for $(t_1,t_2)=(t_1^k,t_2^k)$ that is not satisfied with the inequality. Thus wde add the constraint
	\begin{equation*}
		z\leq f(y^k)+t_1g(y^k)+t_2h(y^k)
	\end{equation*}
	to the constraints in (\ref{CD}) and re-solve the linear program. Thus, the point $(z^k,t_1^k,t_2^k)$ is cut away. 
	
	Note that as the constraints adding to the programming (\ref{CD}), the feasible region is narrowing and the optimal solution values of the problem form a nonincreasing sequence $\{z^k\}$.
	
	The cut-plane algorithm is shown as below
	\begin{algorithm}[h]
		\caption{Cut-Plane}
		\begin{algorithmic}[1]
			\STATE Initialization: given $y^0=(-\frac{x_0}{2};-\frac{x_0}{2})\in D, g(y^0)\leq0, h(y^0)\leq0,k=1$
			\WHILE{$z^k>\theta(t_1^k,t_2^k)$}
			\STATE solve problem (\ref{CD}), let the optimal be $(z^k,t_1^k,t_2^k)$.
			\STATE solve subproblem (\ref{1}), let the optimal of which be $y^k$, and 
			\begin{equation*}
				\theta(t_1^k,t_2^k)= f(y^k)+t_1^kg(y^k)+t_2^kh(y^k).
			\end{equation*}
			\IF{$z^k\leq\theta(t_1^k,t_2^k)$}
			\STATE\label{eps} the algorithm terminates with $(t_1^k,t_2^k)$ be the optimal to the lagrangian dual problem.
			\ELSE
			\STATE add constraint
			\begin{equation*}
				z\leq f(y^k)+t_1g(y^k)+t_2h(y^k)
			\end{equation*}
			to the constraints of (\ref{CD}).
			\STATE $k=k+1$.
			\ENDIF
			\ENDWHILE
			\end{algorithmic}
	\end{algorithm}
	\begin{remark}
		To generate a feasible solution to the primal problem, we add constraints
		\begin{equation}
			g(y^k)\leq0,h(y^k)\leq0
		\end{equation}
	to $(\ref{eps})$ to ensure that $y^k$ is a feasible solution to initial problem $(\ref{ini})$.
	\end{remark}

\subsection{Lagrangian problem}
	Consider the subproblem (\ref{1}), under our model, it can be decoupled into independent minimization subproblems like
	\begin{equation}\label{lag}
		\begin{aligned}
			\min\quad&\nu(y)=y^TH_iy-b_i^Ty\\
			s.t.\quad&y\in D_i=\{(y_1,y_2)|-x_{0,i}\leq y_1,y_2\leq0,y_1+y_2\geq -x_{0,i}\}\\
			&H_i\in\mathbb{S}^{2\times2}
		\end{aligned}
	\end{equation}
	\begin{figure}[h]
		\centering
		\includegraphics[scale=0.6]{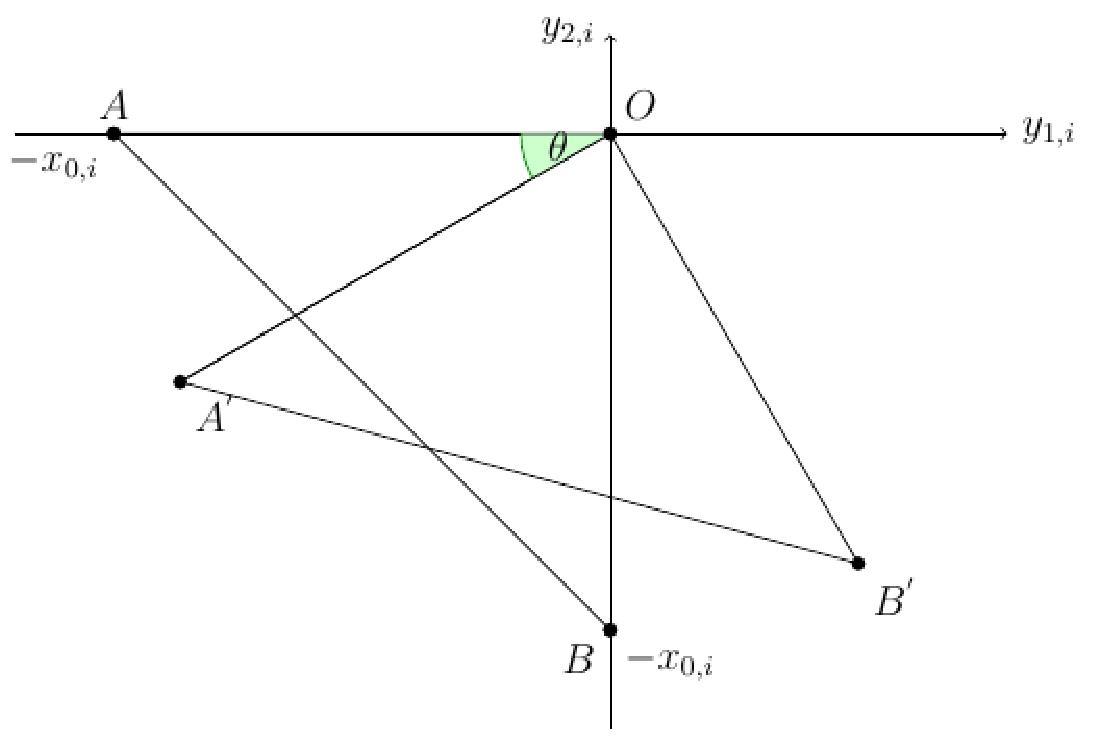}
		\caption{Initial feasible region and the feasible region after rotation transformation}
	\end{figure}

	Since $H_i$ is symetric matrix, we can always find a unit orthogonal matrix $Q_i^TQ_i=I$,
	\begin{equation*}
		Q_i^{T}H_iQ_i=diag(\theta_1,\theta_2)
	\end{equation*}

	 we always have the rotation transformation $y'=Q_iy$ to change the above problem into the following form
	\begin{equation*}
		\begin{aligned}
			\min\quad&u(y)=y^T diag(\theta_1,\theta_2)y+c^Ty\\
			s.t.\quad&y=(y_1,y_2)^T\in\triangle OA^{'}B{'}
		\end{aligned}
	\end{equation*}
	so we need to discuss the distribution of eigenvalues.
	\begin{itemize}
		\item[1.] $\theta_1<0,\theta_2<0$, the problem degenerates into an concave minimization problem, the optimal are avaiable at the vertex $O,A^{'},B^{'}$.
		\item[2.] $\theta_1>0,\theta_2>0$, it's easy to find the explicit solution of the convex problem. For those happen to be located in the feasible region $\triangle OA^{'}B^{'}$, it's sure to be the optimal. If not, the optimal must lie on the boundary of the feasible region, otherwise we can always find the descent direction inside the feasible region towards the explicit solution.
		\item[3.] $\theta_1>0,\theta_2<0$, the problem turns into a D.C. programming. We assert that the optimal lies on the boundary of the feasible region $\triangle OA^{'}B^{'}$.
		
		On the contrary, we assume that $y^*=(y_{1}^*,y_{2}^*)^T\in int \triangle OA^{'}B^{'}$ is the optimal solution. Consider that $\theta_2<0$, the problem is concave about the variable $y_{2}$ and the feasible region is compact, so we can set the descending direction as 
		\begin{equation*}
			e_1=(0,1)^T, e_2=(0,-1)^T.
		\end{equation*}
		such that 
		\begin{equation*}
			\begin{aligned}
				u(y^*+\epsilon e_1)\leq u(y^*) or \\
				u(y^*+\epsilon e_2)\leq u(y^*).
			\end{aligned}
		\end{equation*}
	where $\epsilon>0$, which leads to contradictions.
	\item[4.] $\theta_i=0$. The quadratic terms vanish, the corresponding objective function degenerate into linear function. Since linear function are both convex and concave, after rotation transformation, the optimals of  semi-positive and semi-negative cases are lying on the boundary of the feasible region. What's more, this can be avoid in the numerical experiments by finding the explicit solutions with $MATLAB$.
	\end{itemize}

	From above, we've built a surjection by Orthogonal rotation transform, which does not change the eigenvalues of the initial matrix. So the optimal of the transformed problem located in the boundaries of $\triangle OA^{'}B^{'}$ can be converted to the solution of initial problem located in the boundaries of $\triangle OAB$ by inverse mapping. For the convex and concave situation of problem (\ref{lag}), it's easy to show the explicit solution. Let's turn to the situation of non-convex non-concave, i.e., the solution located on the boundaries.
	\begin{itemize}
		\item[1.] $y^*\in OA$, by letting $y_{2,i}=0$, the problem turns into 
		\begin{equation*}
			\begin{aligned}
				\min\quad & H_{i,11}y_{1,i}^2-b_{i,1}y_{1,i}\\
				s.t.\quad & y_{1,i}\in [-x_{0,i},0].
			\end{aligned}
		\end{equation*}
	If $H_{i,11}>0$, the explicit form of optimal can be given as
	\begin{equation*}
		y^*=(y_{1,i}^*,0)^T,\quad y_{1,i}^*=\max(\min(\frac{b_{i,1}}{2H_{i,11}},0),-x_{0,i})
	\end{equation*}
	Otherwise, the optimal results in the endpoint of interval, i.e. $O$ or $A$.
	\item[2.] $y^*\in OB$, by letting $y_{1,i}=0$, the problem turns into 
	\begin{equation*}
		\begin{aligned}
			\min\quad & H_{i,22}y_{2,i}^2-b_{i,2}y_{2,i}\\
			s.t.\quad & y_{2,i}\in [-x_{0,i},0].
		\end{aligned}
	\end{equation*}
	If $H_{i,22}>0$, the explicit form of optimal can be given as
	\begin{equation*}
		y^*=(0,y_{2,i}^*)^T,\quad y_{2,i}^*=\max(\min(\frac{b_{i,2}}{2H_{i,22}},0),-x_{0,i})
	\end{equation*}
	Otherwise, the optimal can be avaiable at $O$ or $B$.
	\item[3.] $y^*\in AB$, by substitute $y_{2,i}=-x_{0,i}-y_{1,i}$ into the objective function, we have
	\begin{equation*}
		\begin{aligned}
			\min\quad & (H_{i,11}+H_{i,22}-2H_{i,12})y_{i,1}^2-[b_{i,1}-b_{i,2}+2x_{0,i}(H_{i,12}-H_{i,22})]y_{i,1}+c\\
			s.t.\quad & y_{1,i}\in[-x_{0,i},0],\\
			&c=H_{i,22}x_{0,i}^2+b_{i,2}x_{0,i}.
		\end{aligned}
	\end{equation*}
	The same as before, when $H_{i,11}+H_{i,22}-2H_{i,12}>0$, the explicit form of optimal can be shown as
	\begin{equation*}
		\begin{aligned}
			y^*&=(y_{1,i}^*,-x_{0,i}-y_{1,i}^*)\\
			y_{1,i}^* &= \max(\min(y_{exp},0),-x_{0,i}),\quad y_{exp}=\frac{b_{i,1}-b_{i,2}+2x_{0,i}(H_{i,12}-H_{i,22})}{2(H_{i,11}+H_{i,22}-2H_{i,12})}.
		\end{aligned}		
	\end{equation*}
	Otherwise, the optimal should be the endpoint $A$ or $B$.
	\end{itemize}
	
	We conclude above as following proposition
	\begin{proposition}
		The optimal solution to the subproblem (\ref{lag}) can be shown as
		\begin{equation}
			y^*=\left\{\begin{aligned}
			&\arg\min\{\nu(O),\nu(A),\nu(B)\},& &H_i\prec 0,\\
			&\frac{1}{2}H_i^{-1}b_i,& &H_i\succ 0,\frac {1}{2}H_i^{-1}b_i\in D_i,\\
			&
			\begin{split}
				\arg\min\{&\nu(O),\nu(A),\nu(B),\\
				&\nu(y_{OA}),\nu(y_{OB}),\nu(y_{AB})\}
			\end{split} & & else.
			\end{aligned}\right.
		\end{equation}
	where 
	\begin{equation*}
		\begin{aligned}
			y_{OA} &= (y_{temp1}, 0)^T,\quad y_{temp1}=\max(\min(\frac{b_{i,1}}{2H_{i,11}},0),-x_{0,i}),\\
			y_{OB} &=(0,y_{temp2})^T,\quad y_{temp2}=\max(\min(\frac{b_{i,2}}{2H_{i,22}},0),-x_{0,i}),\\
			y_{AB} &= (y_{temp3},-x_{0,i}-y_{temp3})^T,\\
			y_{temp3} &= \max(\min(y_{exp},0),-x_{0,i}),\quad y_{exp}=\frac{b_{i,1}-b_{i,2}+2x_{0,i}(H_{i,12}-H_{i,22})}{2(H_{i,11}+H_{i,22}-2H_{i,12})}.
		\end{aligned}
	\end{equation*}
	\end{proposition}
	As we've done above, we discuss the solution of (\ref{lag}) in detail, and almost give the explicit solution, which results in great imporve in solving the Lagrangian subproblem.

\section{Convex relaxation}
\subsection{Estimate of the upper bound}
	To make the problem more concise, after some notation, we convert the initial problem (\ref{ini}) into the form as below
	\begin{equation}\label{17}
		\begin{aligned}
				\min\quad& f(y)=y^T Q_0 y + q_0^Ty + c_0\\
				s.t.\quad& y^T Q_1 y + q_1^Ty + c_1\leq0,\\
				& y^T Q_2 y + q_2^Ty + c_2\leq0,\\
				&y = (y_1^T, y_2^T)^T\in D.\\
				&D=\{(y_1,y_2)|-x_0\leq y_1,y_2\leq0,y_1+y_2\geq-x_0.\}
			\end{aligned}
	\end{equation}
		where
		\begin{equation*}
				\begin{aligned}
						Q_0&=\begin{bmatrix}
								\Lambda - \frac{\Gamma}{2} &-\frac{\pi}{2}\Gamma\\
								-\frac{\pi}{2}\Gamma & \pi(\Lambda - \frac{\Gamma}{2})
							\end{bmatrix}, q_0=-\begin{bmatrix}
							\Gamma x_0\\
							\pi\Gamma x_0
						\end{bmatrix}, c_0 = \pi\delta-e_0.\\
				Q_1 &= \begin{bmatrix}
						\rho_1(\Lambda-\frac{\Gamma}{2})+\Lambda+\frac{\Gamma}{2} &O\\
						O &O
					\end{bmatrix}, q_1 = \begin{bmatrix}
					p_0-\rho_1\Gamma x_0\\
					O
				\end{bmatrix},c_1 = l_0-\rho_1 e_0.\\
				Q_2 &= \begin{bmatrix}
						\Lambda+\Gamma/2+\rho_2(\Lambda-\frac{\Gamma}{2}) & \frac{1-\rho_2}{2}\Gamma\\
						\frac{1-\rho_2}{2}\Gamma & \Lambda+\Gamma/2+\rho_2(\Lambda-\frac{\Gamma}{2}),\\
					\end{bmatrix}\\
					q_2&=\begin{bmatrix}
							p_0-\rho_2\Gamma x_0\\
							p_0-\rho_2\Gamma x_0
						\end{bmatrix}, c_2=l_0-\rho_2e_0+(\rho_2+1)\delta.
						\end{aligned}
							\end{equation*}
	Let's first recall the Spectral Decomposition theorem.
	\begin{theorem}
		Let $M$ be a real symmetric $n\times n$ matrix with eigenvalues $v_1,...,v_n$ and corresponding orthonormal eigenvectors $\boldsymbol{u}_1,...,\boldsymbol{u}_n$,then:
		\begin{equation*}
			M=\sum_{i=1}^{n}v_i\boldsymbol{u}_i\boldsymbol{u}_i^T.
		\end{equation*}
	\end{theorem}
	With the theorem above, we have the following decomposition
	\begin{equation}
		\begin{aligned}
			Q_0=A^+-A^-,A^+=\sum_{i=p+1}^{2m}a_i\zeta_i\zeta_i^T,A^-=\sum_{i=1}^{p}|a_i|\zeta_i\zeta_i^T,\\
			Q_1=B^+-B^-,B^+=\sum_{i=q+1}^{2m}b_i\eta_i\eta_i^T,B^-=\sum_{i=1}^{q}|b_i|\eta_i\eta_i^T,\\
			Q_2=C^+-C^-,C^+=\sum_{i=r+1}^{2m}c_i\theta_i\theta_i^T,C^-=\sum_{i=1}^{r}|c_i|\theta_i\theta_i^T.
		\end{aligned}
	\end{equation}
	where
	\begin{equation*}
		\begin{aligned}
			a_i<0,i=1,...p,a_i\geq0,i=p+1,...,2m,\\
			b_i<0,i=1,...q,b_i\geq0,i=q+1,...,2m,\\
			c_i<0,i=1,...r,c_i\geq0,i=r+1,...,2m.
		\end{aligned}
	\end{equation*}
	$a_i, b_i, c_i$ and $\zeta_i, \eta_i, \theta_i$ are eigenvalues and corresponding eigenvectors of the matrix.
	
	It's easy to see that $A^-,B^-,C^-$ are positive semidefinite matrix. So for every $u\in D$, we always have
	\begin{equation*}
		\begin{aligned}
		(y-u)^TA^-(y-u)&\geq0\\
		y^TA^-y&\geq2u^TA^-y-u^TA^-u
	\end{aligned}
	\end{equation*}
	the above estimators happens to be a supporting hyperplane of $y^TA^-y$ at $u$. So as $B^-,C^-$. Then we use the linear terms to approximate the concave quadratic terms in the objective and constraints of the problem (\ref{17}), and the relexation problem can be reformulated as
	\begin{equation}\label{cr}
		\begin{aligned}
			\min\quad& f_u(y)=y^T A^+ y -2u^TA^-y+ q_0^Ty +u^TA^-u  + c_0\\
			s.t.\quad& y^T B^+ y -2u^TB^-y+ q_1^Ty +u^TB^-u+ c_1\leq0,\\
			& y^T C^+ y -2u^TC^-y+ q_2^Ty +u^TC^-u+ c_2\leq0,\\
			&y = (y_1^T, y_2^T)^T\in D.\\
			&D=\{(y_1,y_2)|-x_0\leq y_1,y_2\leq0,y_1+y_2\geq-x_0.\}
		\end{aligned}
	\end{equation}
	We denote the feasible region of the initial problem (\ref{P}) and the above convex relaxation problem as 
	\begin{equation*}
		\begin{aligned}
		F&=\{y\in D| g(y)\leq0,h(y)\leq0\},\\
		F_u&=\{y\in D| g_u(y)\leq0,h_u(y)\leq0\}.
				\end{aligned}
	\end{equation*}

	Relying on the convex relaxation problem (\ref{cr}), we proposed the corresponding SCO algorithm\cite{lhz2023} and show some important properties.
	
	\begin{breakablealgorithm}
		\caption{SCO}\label{sco_algo}
		\begin{algorithmic}
			\STATE Initialization: find feasible point $y^0\in F$ and the torlence $\epsilon$.
			\STATE Set $u^0=y^0$,k=0.
			\IF{$\Vert u^{k+1}-u^k\Vert\geq\epsilon$}
			\STATE Solve problem(\ref{cr}) with $u = u^k$ to find the optimal solution $y^{k+1}$. Then set $u^{k+1}=y^{k+1}$.
			\STATE $k=k+1$.
			\ELSE
			\STATE	Output $y^{k+1}$ as the final solution.
			\ENDIF
		\end{algorithmic}
	\end{breakablealgorithm}
	
	\begin{proposition}
		\begin{itemize}
			\item[(i)] $F_u$ is a convex subset of $F$.
		\end{itemize}
	\end{proposition}

	\begin{proposition}
		The sequence ${y^k}$ and ${u^k}$ are generated or chosen by algorithm $\ref{sco_algo}$ are satisfying that
		\begin{itemize}
			\item [$(1)$] ${y^k}\subseteq F$.
			\item [$(2)$] $f(y^k)-f(y^{k+1})\geq (u^{k+1}-u^{k})^TA^-(u^{k+1}-u^{k})$.
		\end{itemize}
	\end{proposition}

	\begin{proof}
		\begin{itemize}
			\item [(1)] From Cut-Plane Algorithm, $y^0\in F$ and $u^0=y^0\in F$, by the construction of convex relaxation, we have
			\begin{equation*}
				\begin{aligned}
					g_u(y) &= g(y) + (y-u)^TB^-(y-u),\\
					h_u(y) &= h(y) + (y-u)^TC^-(y-u).
				\end{aligned}
			\end{equation*}
			so $g_{u^0}(y^0)=g(y^0),h_{u^0}(y^0)=h(y^0)$, then $y^0\in F_{u^0},F_{u^0}\neq\emptyset$. By the algorithm we proposed above, $y^1\in F_{u^0}\subseteq F$, we reach
			\begin{equation*}
				\{y^k\}\in F.
			\end{equation*}
		\item[(2)] From above, we have ${y^k}\subseteq F$. By the algorithm, it's easy to verify that $y^k\in F_{u^k}$. Notice that by construction 
		\begin{equation*}
			f_u(y) = f(y) + (y-u)^TA^-(y-u).
		\end{equation*}
		Since $y^{k+1}$ is the optimal solution of convex relaxation problem (\ref{cr}) with $u = u^{k}$, by $y^k\in F_{u^k}$, we have
		\begin{equation*}
			\begin{aligned}
				f_{u^k}(y^{k+1})&\leq f_{u^k}(y^{k})\\
			f(y^{k+1})+(y^{k+1}-u^k)^TA^-(y^{k+1}-u^k)&\leq f(y^k) + (y^k-u^k)^TA^-(y^k-u^k).
						\end{aligned}
		\end{equation*}
		from the fact that $u^k=y^k$, we have 
		\begin{equation*}
			f(y^k) - f(y^{k+1})\geq (u^{k+1}-u^k)^TA^-(u^{k+1}-u^k)
		\end{equation*}
		\end{itemize}
	\end{proof}

\subsection{Estimate of the lower bound}
	For positive semi-definite matrix $M$, we have the following observation
	\begin{equation*}
		y^TMy\leq \lambda_{max}(M)y^Ty.
		y^TMy\leq \Vert M \Vert_2y^Ty.
	\end{equation*}
	where $\lambda_{max}$ is the max eigenvalue of matrix $M$.

	With this fact, we facilitate (\ref{17}) as 
	\begin{equation}
		\begin{aligned}
			\min\quad& y^T A^+ y -\lambda_{max}(A^-)y^Ty+ q_0^Ty + c_0\\
			s.t.\quad& y^T B^+ y -\lambda_{max}(B^-)y^Ty+ q_1^Ty + c_1\leq0,\\
			& y^T C^+ y -\lambda_{max}(C^-)y^Ty+ q_2^Ty + c_2\leq0,\\
			&y = (y_1^T, y_2^T)^T\in D.\\
			&D=\{(y_1,y_2)|-x_0\leq y_1,y_2\leq0,y_1+y_2\geq-x_0.\}
		\end{aligned}
	\end{equation}
	Substitute the negative quadratic term $-y^Ty$ with its convex envelope, we can reformulate the problem as a convex relaxation of (\ref{17}), which accounts for the estimate of the lower bound of initial problem. The relaxation problem can be described as 
	\begin{equation}
		\begin{aligned}
			\min\quad& y^T A^+ y -\lambda_{max}(A^-)\sum_{i=1}^{2m}t_i+ q_0^Ty + c_0\\
			s.t.\quad& y^T B^+ y -\lambda_{max}\sum_{i=1}^{2m}t_i+ q_1^Ty + c_1\leq0,\\
			& y^T C^+ y -\lambda_{max}(C^-)\sum_{i=1}^{2m}t_i+ q_2^Ty + c_2\leq0,\\
			&y = (y_1^T, y_2^T)^T\in D,\\
			&D=\{(y_1,y_2)|-x_0\leq y_1,y_2\leq0,y_1+y_2\geq-x_0\},\\
			&y_i^2\leq t_i, t_i\leq(l_i+u_i)y_i-l_iu_i,\\
			&[l_i,u_i]\subseteq[-x_{0,i},0],i=1,...2m.
		\end{aligned}
	\end{equation}

\subsection{Estimate of the lower bound through McCormick envelope}
	Before we construct the convex relaxation aiming for the estimate of  lower bound, we first explore a well-known envelope using in relaxation problem.
	\begin{proposition} McCormick envelope can be presented as following forms
		\begin{equation*}
			\begin{aligned}
				\omega &= xy,\quad\underline{x}\leq x\leq \bar{x}, \underline{y}\leq y\leq\bar{y},\\
				\omega&\geq \underline{x}y+x\underline{y}-\underline{x}\underline{y},\quad \omega\geq\bar{x}y+x\bar{y}-\bar{x}\bar{y},\\
				\omega&\leq\bar{x}y+x\underline{y}-\bar{x}\underline{y},\quad\omega\leq x\bar{y}+\underline{x}y-\underline{x}\bar{y}.
			\end{aligned}
		\end{equation*}
	\end{proposition}
	The derivation is obvious. As for quadratic form $x^TAx, A\in S^n,A\succeq0$, we can also show similar propositions.
	\begin{proposition}
		$\forall A\in S^n,A\succeq0$, we give the proper overestimator of $x^TAx,x\in[\underline{x},\bar{x}]$ as following
			\begin{equation}\label{lowerbound}
			x^TAx\leq(\bar{x}+\underline{x})^TAx-\bar{x}^TA\underline{x}.
		\end{equation}
	\end{proposition}
	\begin{proof}
		\begin{equation*}
			\begin{aligned}
				F(x)&=-x^TAx+(\bar{x}+\underline{x})^TAx-\bar{x}^TA\underline{x},x\in[\underline{x},\bar{x}],\\
				&=(\bar{x}-x)^TA(x-\underline{x}),\\
				F'(x)&=A(\bar{x}+\underline{x}-2x),\\
				F''(x)&=-2A.
			\end{aligned}
		\end{equation*}
	Since the objective function is concave, and the KKT point is located in the interval, so the minimum is available at the endpoint of the interval, the maximum is available at the KKT point. And we have
	\begin{equation*}
		F(\bar{x})=F(\underline{x})=0\leq F(x)\leq F(\frac{\bar{x}+\underline{x}}{2}).
	\end{equation*}
	After some algebra, we finally reach
	\begin{equation*}
		\begin{aligned}
			x^TAx&\leq(\bar{x}+\underline{x})^TAx-\bar{x}^TA\underline{x},\\
			x^TAx&\geq(\bar{x}+\underline{x})^TAx-(\frac{\bar{x}+\underline{x}}{2})^TA(\frac{\bar{x}+\underline{x}}{2}).
		\end{aligned}
	\end{equation*}
	The latter inequality happens to be the supporting hyperplane at $x=\frac{\bar{x}+\underline{x}}{2}$.
	\end{proof}

	With the inequality (\ref{lowerbound}), we can reformulate the problem as a convex relaxation of (\ref{17}), which accounts for the estimate of the lower bound of initial problem. The relaxation problem can be described as
	\begin{equation}\label{lb_convex_relaxation}
		\begin{aligned}
			\min\quad& y^T A^+ y -((\bar{y}+\underline{y})^TA^-y-\bar{y}^TA^-\underline{y})+ q_0^Ty + c_0\\
			s.t.\quad& y^T B^+ y -((\bar{y}+\underline{y})^TB^-y-\bar{y}^TB^-\underline{y})+ q_1^Ty + c_1\leq0,\\
			& y^T C^+ y -((\bar{y}+\underline{y})^TC^-y-\bar{y}^TC^-\underline{y})+ q_2^Ty + c_2\leq0,\\
			&y= (y_1^T, y_2^T)^T\in D.\\
			&D=\{(y_1,y_2)|-x_0\leq y_1,y_2\leq0,y_1+y_2\geq-x_0.\},\\
			&y\in[\underline{y},\bar{y}].
		\end{aligned}
	\end{equation}
	
	\begin{theorem}
		Suppose that $y^*$ is the optimal solution of problem (\ref{lb_convex_relaxation}), and the corresponding optimal value is $v^*_{[\underline{y},\bar{y}]}$, then we have
		\begin{equation*}
			\begin{aligned}
			f(y^*)-v^*_{[\underline{y},\bar{y}]}\leq\frac{1}{4}\Vert A^-\Vert_2\Vert\bar{y}-\underline{y}\Vert_2^2,\\
			h(y^*)\leq\frac{1}{4}\Vert C^-\Vert_2\Vert\bar{y}-\underline{y}\Vert_2^2.
		\end{aligned}
		\end{equation*}
	\end{theorem}
	\begin{proof}
		\begin{equation*}
			\begin{aligned}
				f(y^*)-v^*_{[\underline{y},\bar{y}]}&=-(y^*)^TA^-y^* +(\bar{y}+\underline{y})^TA^-y^*-\bar{y}^TA^-\underline{y},\\
				&=(\bar{y}-y^*)^TA^-(y^*-\underline{y}),\\
				&\leq(\frac{\bar{y}-\underline{y}}{2})^TA^-(\frac{\bar{y}-\underline{y}}{2}),\\
				&\leq\frac{1}{4}\Vert A^-\Vert_2\Vert\bar{y}-\underline{y}\Vert_2^2.
			\end{aligned}
		\end{equation*}
	Similarly, the optimlity condition
	\begin{equation*}
		\begin{aligned}
			h(y^*)&=(y^*)^TC^+y^* - (y^*)^TC^-y^*+q^T_2y^*+c_2,\\
			&=(y^*)^T C^+ y^* -((\bar{y}+\underline{y})^TC^-y^*-\bar{y}^TC^-\underline{y})+ q_2^Ty^* + c_2\\
			&+(\bar{y}-y^*)^TC^-(y^*-\underline{y}),\\
			&\leq(\bar{y}-y^*)^TC^-(y^*-\underline{y}),\\
			&\leq(\frac{\bar{y}-\underline{y}}{2})^TC^-(\frac{\bar{y}-\underline{y}}{2})\\
			&\leq\frac{1}{4}\Vert C^-\Vert_2\Vert\bar{y}-\underline{y}\Vert_2^2.
		\end{aligned}
	\end{equation*}
	\end{proof}

\section{Algorithm and Analysis}
	Based on the convex relaxation problem (\ref{lb_convex_relaxation}) aiming at the estimate of lower bound, we propose a global algorithm, which combines the SCO algorithm and the B$\&$B framework. The algorithm is described as below.
	
	\begin{breakablealgorithm}
		\caption{SCOBB}
		\begin{algorithmic}
			\STATE \textbf{Input}:$\Lambda, \Gamma, p_0, x_0, \rho_1, \rho_2, \pi, \delta$ and tolerance $\epsilon$.
			\STATE \textbf{Output}: $\epsilon-$optimal solution $y^*$.
			\STATE \textbf{Initialization}: 
			\STATE (i)Compute $A^+,A^-,B^+,B^-,C^+,C^-$ by SVD.
			\STATE (ii) Set the initial bounds as $\bar{y}^0=0,\underline{y}^0=-x_0$.
			\STATE (iii) Calculate an initial point $y_0$ by Cut-Plane algorithm to serve for SCO algorithm.
			\STATE \textbf{Step 1}
			\STATE Find a point $y^*$ of problem (\ref{17}) by using the SCO algorithm. Set the first upper bound $v^*=f(y^*)$.
			\STATE \textbf{Step 2}
			\STATE Solve problem (\ref{lb_convex_relaxation}) over initial bounds to obtain the optimal solution $y^0$ and the first lower bound $v^0$. If $g(y^0)\leq0,h(y^0)\leq\epsilon$ and $f(y^0)<v^*$, then update upper bound $v^*=f(y^0)$ and solution $y^*=y^0$. Set $k=0, \Delta^k:=[\underline{y}^k,\bar{y}^k],\Omega:=\{[\Delta^k,v^k,y^k]\}$.
			\STATE \textbf{Step 3}
			\WHILE{$\Omega\neq\emptyset$} 
			\STATE (\textbf{Node Selection}) Choose a node $[\Delta^k,v^k,y^k]$ from $\Omega$ with the smallest lower bound $v^k$ and delete it from $\Omega$.
			\STATE (\textbf{Termination}) If $v^k\geq v^* -\epsilon$, then $y^*$ is an $\epsilon-$optimal solution to original problem, stop.
			\STATE (\textbf{Partition}) Choose $i^*$ as the index of the maximum value of $\bar{y}^k_i-\underline{y}^k_i$ for $i=1,...2m$. Then set the branch point $\beta_{i^*}=\frac{\underline{y}^k_{i^*}+\bar{y}^k_{i^*}}{2}$. Partition $\Delta^k$ into two sub-rectangles $\Delta^{k_1}$ and $\Delta^{k_2}$ along the edge $[\underline{y}^k_{i^*},\bar{y}^k_{i^*}]$ at point $\beta_{i^*}$.
			\STATE For $j=1,2$, solve problem (\ref{lb_convex_relaxation}) over $\Delta^{k_j}$ if feasible to get the optimal solution $y^{k_j}$ and the optimal value $v^{k_j}$. Set $\Omega = \Omega\cup\{[\Delta^{k_1},v^{k_1},y^{k_1}]\}\cup\{[\Delta^{k_2},v^{k_2},y^{k_2}]\}$.
			\STATE \textbf{Restart SCO}
			\STATE Set $\hat{y}=\arg\min\{f(y^{k_1}),f(y^{k_2})\}$. If $g(\hat{y})<0,h(\hat{y})<\epsilon$ and $f(\hat{y})\leq v^*-\epsilon$, then find a solution $\bar{y}^k$ of problem (\ref{17}) by using the SCO algorithm with the initial $u = \hat{y}$, update solution $y^* = \arg\min\{f(\hat{y}),f(\bar{y}^k)\}$ and upper bound $v^*=f(y^*)$.
			\STATE \textbf{Node Deletion}
			\STATE Delete all nodes $[\Delta^j,v^j,y^j]$ with $v^j\geq v^*-\epsilon$ from $\Omega$. Set $k=k+1$.
			\ENDWHILE
		\end{algorithmic}
	\end{breakablealgorithm}

\begin{theorem}
	Algorithm 3 can find an $\epsilon-$optimal solution to problem (P) via solving at most \[\prod_{i=1}^{2m}\lceil\frac{\Vert C^-\Vert_2(\bar{y}^i-\underline{y}^i)}{2\sqrt{\epsilon}}\rceil\] convex relaxation subproblem (\ref{lb_convex_relaxation}).
\end{theorem}
\newpage
\section{Numerical Experiments}
\begin{table}[!ht]
	\caption{Performances of SCOBB and SCO algorithms with $\rho_1=\rho_2=18$ and $\epsilon_{rel} = 10^{-6}$}
	\begin{tabular}{lccc|cccc|cccc}
		\hline
		\multicolumn{4}{c|}{Dimens./Parmas} & \multicolumn{4}{c|}{SCOBB} & \multicolumn{4}{c}{SCO}\\
		\hline
		&$m$ & $\pi$ & $\frac{\delta}{\delta_{\max}}$ & OptVal & Time & Iter & $\frac{l_2(y^*)}{e_2(y^*)}$ & Gap & Time & Iter & $\frac{l_2(y^*)}{e_2(y^*)}$\\
		\hline
		&100 & 0.3 & 0.8  &-9353.7215	 &5.21	 &32.0	 &18.0000  &0	 &0.73	 &14.5	 &18.0000\\
		&200 & 0.3 & 0.8  &-18790.7526	 &13.31	 &34.2	 &18.0000  &0	 &2.28	 &36.8	 &18.0000\\
		&500 & 0.3 & 0.8  &-47422.2786	 &52.62	 &22.4	 &18.0000  &0	 &28.58	 &102.1	 &18.0000\\							
		\hline
	\end{tabular}
\end{table}
	\begin{remark}
		Since the nodes picked by SCOBB algorithm in the solution pool is random, we can not guarantee that the iter is positively correlated with the dimensions; What's more, Gurobi cannot reach optimal unless the relative tolerance is set to $10^{-4}$.
	\end{remark}
\section{Promotion}
	Let's consider such QCQP problem
	\begin{equation}
		\begin{aligned}
			\min\quad &x^TQ_0x+x^Tq_0+c_0\\
			s.t.\quad &x^TQ_ix+x^Tq_i+c_i\leq0, i=1,\dots,m.\\
					  &x\in[\underline{x},\bar{x}].
		\end{aligned}
	\end{equation}
	where $Q_i\in \mathbb{S}^{n\times n}, i = 0,1,\dots,m$.
\begin{assumption}
	There exist $x_0$ in $(\underline{x},\bar{x})$, such that 
	\begin{equation*}
		x^TQ_ix+x^Tq_i+c_i<0, i=1,\dots,m.
	\end{equation*}
\end{assumption}
	The slater condition is plain to ensure that the problem is at least feasible. And Except this we don't make any convexity assumption.
	
	Since $Q_i$ is symmetric, so we can always find eigenvalues of each matrix, and use the spectral decomposition to decompose the matrix into
	\begin{equation*}
		Q_i = Q_i^+-Q_i^-
	\end{equation*}	
	As we've done before, we show the upper bound estimators by the convex relaxation problem
	\begin{equation}
		\begin{aligned}
			\min\quad & x^TQ_0^{+}x-2u^TQ_0^{-}x+q^T_0x+u^TQ_0^-u+c_0\\
			s.t.\quad & x^TQ_i^{+}x-2u^TQ_i^{-}x+q^T_ix+u^TQ_i^-u+c_i\leq0,i=1,\dots,m.\\
				&x\in[\underline{x},\bar{x}].
		\end{aligned}
	\end{equation}
	And then we apply the above convex relaxation problem to $SCO $ algorithm to generate the upper bound.
	By the derivation, we know that $u$ can be any point. However, we still recommend a feasible point since it'll accelerate the algorithm convergence speed.
	And the corresponding lowerbound convex relaxation problem as
	\begin{equation}
		\begin{aligned}
			\min\quad & x^TQ_0^+x-(\underline{x}+\bar{x})^TQ_0^-x+q^T_0x+\underline{x}Q_0^-\bar{x}+c_0\\
			s.t.\quad &x^TQ_i^+x-(\underline{x}+\bar{x})^TQ_i^-x+q^T_ix+\underline{x}Q_i^-\bar{x}+c_i\leq0,i=1,\dots,m.\\
			&x\in[\underline{x},\bar{x}].
		\end{aligned}
	\end{equation}
	Similarly, we use the $SCOBB $ algorithm to find the lower bound.
	
	Compare with Lu\cite{lu2019}, our algorithm is not sensitive with the amount of negative eigenvalues of single matrix and the quadratic constraints. And we don't need the initial point to start the $SCO$ algorithm. 

\bibliographystyle{unsrt}
\bibliography{ref}

\end{document}